\documentclass{conm-p-l}

\usepackage{}

\newtheorem{theorem}{Theorem}[section]

\theoremstyle{definition}
\newtheorem{definition}[theorem]{Definition}
\newtheorem{example}[theorem]{Example}

\theoremstyle{remark}
\newtheorem{remark}[theorem]{Remark}

\numberwithin{equation}{section}
\newcommand{\RR}{\mathbb R}
\newcommand{\norm}[1]{\|#1\|}
\newtheorem{prob}[theorem]{Problem}
\newtheorem{notation}[theorem]{Notation}
\newtheorem{corollary}[theorem]{Corollary}
\newtheorem{ex}[theorem]{Example}
\newtheorem{proposition}[theorem]{Proposition}
\newcommand{\Q}{\mathbb{Q}}
\DeclareMathOperator{\spn}{span}
\begin{document}

\title{Phase Retrieval by Hyperplanes}


\author[Botelho-Andrade]{Sara Botelho-Andrade}
\address{Department of Mathematics, University
	of Missouri, Columbia, MO 65211-4100}
\curraddr{}
\email{sandrade102087@gmail.com}
\thanks{The first 6 authors were supported by
	NSF DMS 1307685; and  NSF ATD 1321779; ARO W911NF-16-1-0008.
Zhiqiang Xu was supported  by NSFC grant (11422113,  91630203, 11331012).
	Part of this reseaech was carried out while the authors
	were visiting the Hong Kong University of Science and
	Technology with support from a grant from (ICERM)
	Institute for Computational and Experimental Research in
	Mathematics.}

\author[Casazza]{Peter G. Casazza}
\address{Department of Mathematics, University
	of Missouri, Columbia, MO 65211-4100}
\curraddr{}
\email{Casazzap@missouri.edu}
\thanks{}

\author[Cheng]{Desai Cheng}
\address{Department of Mathematics, University
	of Missouri, Columbia, MO 65211-4100}
\curraddr{}
\email{chengdesai@gmail.com}
\thanks{}

\author[Haas]{John Haas}
\address{Department of Mathematics, University
	of Missouri, Columbia, MO 65211-4100}
\curraddr{}
\email{terraformthedreamscape@gmail.com}
\thanks{}

\author[Tran]{Tin T. Tran}
\address{Department of Mathematics, University
	of Missouri, Columbia, MO 65211-4100}
\curraddr{}
\email{tinmizzou@gmail.com}
\thanks{}

\author[Tremain]{Janet C. Tremain}
\address{Department of Mathematics, University
	of Missouri, Columbia, MO 65211-4100}
\curraddr{}
\email{Tremainjc@missouri.edu}
\thanks{}

\author[Xu]{Zhiqiang Xu}
\address{LSEC, Inst. Comp. Math., Academy of Mathematics and System Science, Chinese Academy of Sciences, Beijing, 100091, China
	}
\curraddr{}
\email{xuzq@lsec.cc.ac.cn}
\thanks{}


\date{}

\begin{abstract}
	We show that a scalable frame does phase retrieval
	if and only if
	the hyperplanes of its orthogonal complements do phase
	retrieval.  We then show this result
	fails in general by giving an example of a frame for $\RR^3$ which does phase
	retrieval but its induced hyperplanes fail phase retrieval. Moreover, we show that such frames always exist in $\RR^d$ for any dimension $d$. We also
	give an example of a frame in $\RR^3$ which fails phase retrieval
	but its perps do phase retrieval. We will also see that a family of
	hyperplanes doing phase retrieval in $\RR^d$ must contain at least
	$2d-2$ hyperplanes.  Finally, we provide an example of six
	hyperplanes in $\RR^4$ which do phase retrieval.
\end{abstract}

\maketitle

\section{Introduction}

In some applications in engineering, the phase of a signal is
lost during processing.
The problem of retrieving the phase of a signal, given a set of intensity measurements, has been studied by engineers for many years. Signals passing through linear systems often result in lost or distorted phase information. This partial loss of phase information occurs in various applications including speech recognition~\cite{BeRi99,RaJu93,ReBlScCa004}, and optics applications such as X-ray crystallography~\cite{BaMn86,Fi78,Fi82}. The concept of \textit{phase retrieval} for Hilbert space
frames was introduced in 2006 by Balan, Casazza, and Edidin~\cite{BCE} and since then it has become an active area of research. Phase retrieval deals with recovering the phase of a signal given intensity measurements from a redundant linear system. In phaseless reconstruction the unknown signal itself is reconstructed from these measurements. In recent literature, the two terms were used interchangeably.  However it is not obvious from the definitions that the two are equivalent. Recently, authors in~\cite{SaCa016} proved that phase retrieval is equivalent to phaseless reconstruction in both the real and complex case.

Phase retrieval has been defined for vectors as well as for projections. \textit{Phase retrieval by projections} occur in real life problems, such as crystal twinning~\cite{Dr010}, where the signal is projected onto some higher dimensional subspaces and has to be recovered from the norms of the projections of the vectors onto the subspaces. We refer the reader to~\cite{CCPW} for a detailed study of phase retrieval by projections. At times these projections are identified with their target spaces. Determining when subspaces $\{W_i\}_{i=1}^n$ and $\{W_i^\perp\}_{i=1}^n$ both do phase retrieval has given way to the notion of \textit{norm retrieval} \cite{BaCaCaJaWo014}, another important area of research.

In this paper we make a detailed study of phase retrieval by
hyperplanes.  We will see that it takes at least $2d-2$
hyperplanes to do phase retrieval in $\RR^d$.  We will show
that scalable frames $\{\phi_i\}_{i=1}^n$ do phase retrieval if and only if their induced hyperplanes $\{\phi_i^{\perp}\}_{i=1}^n$
do phase retrieval.  We then give examples to show this result
fails in general if the frame is not scalable. In particular, we give an example of a frame for $\RR^3$ which does phase
retrieval but its induced hyperplanes fail phase retrieval. Moreover, we show that such frames always exist in $\RR^d$ for any dimension $d$.
We also give an example of a family of hyperplanes in $\RR^3$
which do phase retrieval but their perp vectors fail phase
retrieval.  Finally, we give 6 hyperplanes in $\RR^4$ which
do phase retrieval.

\section{Preliminaries}

In this section we will give the background material needed for
the paper.  We start with the definition of a frame.
\begin{definition}\label{D:frame}
	A family of vectors $\Phi=\{\phi_i\}_{i=1}^n$ in $\RR^d$ is a {\bf frame} if there are constants $0<A\leq B<\infty$ so that for all $x\in \RR^d$
	\[A \|x\|^2 \leq \sum_{i=1}^n |\langle x, \phi_i\rangle|^2\leq B\norm{x}^2, \]
	where $A$ and $B$ are the {\bf lower and upper frame bounds} of the frame, respectively. The frame is called an {\bf A-tight frame} if $A=B$ and is a {\bf Parseval frame} if $A=B=1$.
\end{definition}
\begin{definition}\label{D:scalable_frame}
	A frame $\Phi=\{\phi_i\}_{i=1}^n$ in $\RR^d$ is called scalable if there exists scalars $\{s_i\}_{i=1}^n$ such that $\{s_i\phi_i\}_{i=1}^n$ is a tight frame for $\RR^d$.
\end{definition}
The main topics here are phase retrieval and norm retrieval in
$\RR^d$.

\begin{definition}\label{D:phase_ret&phaseless}
	Let $\Phi=\{\phi_i\}_{i=1}^n \subset \RR^d$ be such that for $x, y\in \RR^d$
	\[ |\langle x,\phi_i\rangle|=|\langle y,\phi_i\rangle|, \mbox{ for all }i=1,2,\ldots,n. \]
	$\Phi$ yields
	\begin{itemize}
		\item[(ii)]~ \textbf{phaseless reconstruction} if $x=\pm y$.
		\item[(iii)]~ \textbf{norm retrieval} if $\|x\|=\|y\|$.
	\end{itemize}
\end{definition}
\begin{remark}
	It is easy to see that $\{\phi_i\}_{i=1}^n$ does phase retrieval (norm retrieval) if and only if $\{c_i\phi_i\}_{i=1}^n$ does phase retrieval (norm retrieval), for any non-zero scalars $\{c_i\}_{i=1}^n\subset \RR^d$.
\end{remark}

The paper \cite{BCE} gives the minimal number of vectors needed
in $\RR^d$ to do phase retrieval.

\begin{theorem}[\cite{BCE}]
	In order for a frame $\{\phi_i\}_{i=1}^n$ in $\RR^d$ to do phase
	retrieval, it is necessary that $n\ge 2d-1$.
\end{theorem}
Also  \cite{BCE} presents  a fundamental classification of the
frames which do phase retrieval in $\RR^d$.  For this we need
a definition.

\begin{definition}[\cite{BCE}]\label{D:complement_prop}
	A frame $\Phi=\{\phi_i\}_{i=1}^n $in $\RR^d$ satisfies the {\bf complement property} if for all subsets ${I}\subset\{1, 2, \ldots, n\}$, either $\{\phi_i\}_{i\in I}$ or $\{\phi_i\}_{i\in I^c}$ spans $\RR^d$.
\end{definition}
A fundamental result from \cite{BCE} is:

\begin{theorem}[\cite{BCE}]
	A frame $\Phi$ does phaseless reconstruction in
	$\RR^d$ if and only if it has the complement property.
\end{theorem}

It follows that if $\Phi=\{\phi_i\}_{i=1}^n$ does phase retrieval in $\RR^d$ then $n\ge 2d-1$.  Full spark is another important notion of vectors in frame theory. A formal definition is given below:
\begin{definition}\label{D:full_Spark}
	Given a family of vectors $\Phi=\{\phi_i\}_{i=1}^n$ in $\RR^d$, the {\bf spark} of $\Phi$ is defined as the cardinality of the smallest linearly dependent subset of $\Phi$. When spark$(\Phi) = d + 1$, every subset of size $d$ is linearly independent, and in that case, $\Phi$ is said to be {\bf full spark}.
\end{definition}

We note that from the definitions it follows that full spark frames with $n\ge 2d-1$ vectors have the complement property and hence do phaseless reconstruction.  Also, if $n=2d-1$ then the complement property clearly implies full spark.

We will need a generalization of phase retrieval to phase retrieval
by projections.
\begin{definition}
	A family of subspaces $\{W_i\}_{i=1}^n$ (or respectively, their
	induced projections $\{P_i\}_{i=1}^n$) do
	\begin{enumerate}
	\item {\bf phase retrieval} on
	$\RR^d$ if whenever $x,y \in \RR^d$ satisfy
	\[ \|P_ix\|=\|P_iy\|,\mbox{ for all }i=1,2,\ldots,n,\]
	then $x=\pm y$.
	
\item	It does {\bf norm retrieval} if $\|x\|=\|y\|$.
\end{enumerate}
\end{definition}

We will need a result from \cite{CCPW}.

\begin{proposition}\label{pp2}
Let projections $\{P_i\}_{i=1}^n$ do phase retrieval on $\RR^d$.
Then $\{(I-P_i)\}_{i=1}^n$ does phase retrieval if and only if
it does norm retrieval.
\end{proposition}

We note the following result from \cite{CCPW}:

\begin{theorem}[\cite{CCPW}]
	In $\RR^d$, for any integers $1\le k_i \le d-1$, there are subspaces
	$\{W_i\}_{i=1}^{2d-1}$ of $\RR^d$ with dim $W_i = k_i$ and
	$\{W_i\}_{i=1}^{2d-1}$ does phase retrieval.
\end{theorem}

The major open problem in the area of real phase retrieval is:
\begin{prob}
	What is the least number of subspaces needed to do phase retrieval
	on $\RR^d$?  What are the possible dimensions of these subspaces?
\end{prob}

For notation we will use:
\begin{notation}
	If $\Phi=\{\phi_i\}_{i=1}^n$ is a frame in $\RR^d$, we
	denote the induced hyperplanes as $\Phi^{\perp}=\{\phi_i^{\perp}
	\}_{i=1}^n$.
\end{notation}
\section{Phase Retrieval by Hyperplanes}

We will need a result of Edidin \cite{ED}, which is also generalized in \cite{WX}.

\begin{theorem}\label{edidin}
	Let $\{W_i\}_{i=1}^n$ be subspaces of $\RR^d$ with respective
	projections $\{P_i\}_{i=1}^n$.  The following are equivalent:
	\begin{enumerate}
		\item $\{W_i\}_{i=1}^n$ does phase retrieval.
		\item For every $0\not= x \in \RR^d$,  $\spn\{P_ix\}_{i=1}^n=\RR^d$.
	\end{enumerate}
\end{theorem}
We will show that for a scalable frame $\Phi$, both $\Phi$ and
$\Phi^{\perp}$ do norm retrieval.  For this we need a proposition.

		\begin{proposition}\label{pp1}
			Let $\{W_i\}_{i=1}^n$ be proper subspaces of $\RR^d$ with respective projections $\{P_i\}_{i=1}^n$. Then for any scalars $\{a_i\}_{i=1}^n \subset \RR$ and $0<A\in \RR$, the following are equivalent.
			\begin{enumerate}
				\item For every orthonormal basis $\{u_{i,j}\}_{j=1}^{n_i}$ of $W_i$, the set ${\{a_iu_{i,j}\}_{i=1}^n}_{j=1}^{n_i}$ is a A-tight frame.
				\item For some orthonormal basis $\{u_{i,j}\}_{j=1}^{n_i}$ of $W_i$, the set ${\{a_iu_{i,j}\}_{i=1}^n}_{j=1}^{n_i}$ is a A-tight frame.
				\item $\sum_{i=1}^{n}a_i^2P_i=A\cdot I$.
				\item $\sum_{i=1}^na_i^2(I-P_i)=\left (
				\sum_{i=1}^na_i^2 -A\right )\cdot I.$
				\item For every orthonormal basis $\{v_{i,j}\}_{j=1}^{d-n_i}$ of $W_i^\perp$, the set ${\{a_iv_{i,j}\}_{i=1}^n}_{j=1}^{d-n_i}$ is a $(\sum_{i=1}^na_i^2 -A)$-tight frame.
				\item For some orthonormal basis $\{v_{i,j}\}_{j=1}^{d-n_i}$ of $W_i^\perp$, the set ${\{a_iv_{i,j}\}_{i=1}^n}_{j=1}^{d-n_i}$ is a $(\sum_{i=1}^na_i^2 -A)$-tight frame.		
			\end{enumerate}
		\end{proposition}
		\begin{proof}
		$(1)\Rightarrow (2)$. Obvious.
		
			$(2) \Rightarrow (3)$. Let $\{u_{i,j}\}_{j=1}^{n_i}$ be the orthonormal basis of $W_i$ in (2). Then for any $x\in \RR^d $,
			$$a_i^2P_ix=\sum_{j=1}^{n_i}\langle x, a_iu_{i,j}\rangle a_iu_{i,j}.$$
			Hence $$\sum_{i=1}^na_i^2P_ix=\sum_{i=1}^{n}\sum_{j=1}^{n_i}\langle x, a_iu_{i,j}\rangle a_iu_{i,j}=Ax.$$
			Therefore, $\sum_{i=1}^{n}a_i^2P_i=A\cdot I$.
	
	$(3) \Rightarrow (1)$ Let $\{u_{i,j}\}_{j=1}^{n_i}$ be any orthonormal basis of $W_i$. Then we have
	$$a_i^2P_ix=\sum_{j=1}^{n_i}\langle x, a_iu_{i,j}\rangle a_iu_{i,j}.$$
	Hence $$Ax=\sum_{i=1}^{n}a_i^2P_ix=\sum_{i=1}^{n}\sum_{j=1}^{n_i}\langle x, a_iu_{i,j}\rangle a_iu_{i,j}.$$
	So, ${\{a_iu_{i,j}\}_{i=1}^n}_{j=1}^{n_i}$ is a A-tight frame.
	
	$(3) \Leftrightarrow (4)$. Obvious.
	
	Similarly, $(4), (5), (6)$ are equivalent, but we need to see
that $\sum_{i=1}^na_i^2 -A > 0$. This follows immediately from $  \sum_{i=1}^{n}a_i^2P_i=A\cdot I$ and $\{W_i\}_{i=1}^n$ are proper subspaces.
		\end{proof}

		\begin{proposition}\label{pp3}
			If $\{W_i\}_{i=1}^n$ satisfy one of the conditions in Proposition \ref{pp1}, then both $\{W_i\}_{i=1}^n$ and $\{W_i^\perp\}_{i=1}^n$ do norm retrieval.
		\end{proposition}
		\begin{proof}
			The results follow from the fact that
			$$\sum_{i=1}^{n}a_i^2\| P_ix\|^2=\sum_{i=1}^{n}\langle a_iP_ix, a_iP_ix\rangle=\langle\sum_{i=1}^{n}a_i^2P_ix, x\rangle =A\Vert x\Vert^2.$$
The other case is similar.
		\end{proof}

		\begin{corollary}\label{co1}
				If $\{W_i\}_{i=1}^n$ satisfy one of the conditions in Proposition \ref{pp1}, then $\{W_i\}_{i=1}^n$ does phase retrieval if and only if $\{W_i^\perp\}_{i=1}^n$ does phase retrieval.
		\end{corollary}
		\begin{proof}
			This follows from Proposition \ref{pp2} and Proposition \ref{pp3}.
		\end{proof}
		
	\begin{corollary}\label{co2}
		If $\Phi=\{\phi_i\}_{i=1}^n$ is a scalable frame in $\RR^d$ then $\Phi$ does phase retrieval if and only if $\Phi^\perp$ does phase retrieval.
 	\end{corollary}
 	\begin{proof}
If $P_i$ is the projection onto $\spn\{\phi_i\}$ then for any $x\in \RR^d$,
\[ \|\phi_i\|^2P_ix = \langle x, \phi_i\rangle \phi_i.\]
Since $\Phi$ is scalable then there exist scalars $\{s_i\}_{i=1}^n$ such that $\{s_i\phi_i\}_{i=1}^n$ is a $A$-tight frame.

Therefore, for any $x\in \RR^d$,
 \[Ax=\sum_{i=1}^{n}\langle x, s_i\phi_i\rangle s_i\phi_i =\sum_{i=1}^{n}s_i^2\|\phi_i\|^2P_ix.\]
The result follows by Corollary \ref{co1}.
 	\end{proof}

Now we will give examples to show that the Corollary \ref{co2}
does not hold in general without the assumption the frame
being scalable.  First, let us examine the obvious approach to see why it fails in general. It is known that if $\{\phi_i\}_{i=1}^n$
does phase retrieval and $T$ is an invertible operator then
$\{T\phi_i\}_{i=1}^n$ does phase retrieval.
If $\Phi = \{\phi_i\}_{i=1}^n$ is any frame with frame operator
$S$ which does phase retrieval, $S^{-1/2}\Phi=\{S^{-1/2}\phi_i\}_{i=1}^n$ is a Parseval frame and so
\[ \left \{ \left (S^{-1/2}\phi_i \right )^{\perp}\right \}_{i=1}^n
= \left \{  S^{1/2}\phi_i^{\perp} \right \}_{i=1}^n,\]
does phase retrieval.    So we would like
to apply the invertible operator $S^{-1/2}$ to our hyperplanes
to conclude that $\Phi^{\perp}$ does phase retrieval.  The problem
is that it is known \cite{CCPW} the invertible operators may not
take subspaces doing phase retrieval to subspaces doing phase
retrieval.
\begin{ex}
	There is a frame $\{\phi_i\}_{i=1}^5$ in $\RR^3$ which does phase
	retrieval but the hyperplanes
	$\{\phi_i^{\perp}\}_{i=1}^5$ fail phase retrieval.
\end{ex}
\begin{proof}		Let $\phi_1 = (0, 0, 1), \phi_2 = (1, 0, 1), \phi_3= (0, 1, 1), \phi_4=(1, 1-\sqrt{2}, 2), \phi_5= (1, 1, 1)$.
	
	Since $\{\phi_i\}_{i=1}^5$ is a full spark frame of $5$ vectors in $\RR^3$ then it does phase retrieval.
	
	We have,
	\begin{align*}
	W_1&=\{\phi_1^{\perp}\}=\{(x_1,x_2,x_3)\in \RR^3 : x_3=0\} \\
	W_2&=\{\phi_2^{\perp}\}=\{(x_1,x_2,x_3)\in \RR^3 : x_1+x_3=0\}\\
	W_3&=\{\phi_3^{\perp}\}=\{(x_1,x_2,x_3)\in \RR^3:x_2+x_3=0\}\\
	W_4&=\{\phi_4^{\perp}\}=\{(x_1,x_2,x_3)\in \RR^3:x_1+(1-\sqrt{2})x_2+2x_3=0\}\\
	W_5&=\{\phi_5^{\perp}\}=\{(x_1,x_2,x_3)\in \RR^3:x_1+x_2+x_3=0\}.\\
	\end{align*}
	Let $P_i$ be the orthogonal projection onto $W_i$, then
	\begin{align*}
	P_1(\phi_5)&=(1, 1, 0)\\
	P_2(\phi_5)&=(0, 1, 0)\\
	P_3(\phi_5)&=(1, 0, 0)\\
	P_4(\phi_5)&=(1/2, (1+\sqrt{2})/2, 0)\\
	P_5(\phi_5)&=(0, 0, 0).\\
	\end{align*}
	Thus, $\spn\{P_i(\phi_5)\}_{i=1}^5=W_1\not=\RR^3$. By Theorem \ref{edidin}, $\{W_i\}_{i=1}^5$ cannot do phase retrieval.
\end{proof}

\begin{corollary}
	There exists $\{\phi_i\}_{i=1}^5$  in $\RR^3$ which does phase retrieval but $\{\phi_i^{\perp}\}_{i=1}^5$ cannot do norm retrieval.
\end{corollary}

Now we will generalize this example to all of $\RR^d$.  This
	example looks like it came from nowhere, so we first explain
why this is logical by {\it reverse engineering} the above example
in $\RR^d$.
We need a full spark set of unit vectors $\{\phi_i\}_{i=1}^{2d-1}$
(which therefore do phase retrieval on $\RR^d$) with projections
$P_i$ onto $\spn\{\phi_i\}$, and a vector $x$ so that
$\{(I-P_i)x\}_{i=1}^{2d-1}$ is contained in a hyperplane.  So we
decide in advance that the vector $x$ will be $x=(1,1,\ldots,1)$
and the hyperplane will be
\[ H=\{(c_1,c_2,\ldots,c_{d-1},0):c_i \in \RR\}.\]
Given a $\phi = (a_1,a_2,\ldots,a_d)$ of this type, we have:
\begin{enumerate}
\item We have:
\[ \|\phi\|^2 = \sum_{i=1}^d a_i^2=1.\]
\item We have:
\begin{eqnarray*}
(I-P_i)x &=& (1,1,\ldots,1)- \langle x,\phi\rangle \phi\\
&=& (1,1,\ldots,1)-\left(\sum_{i=1}^da_i\right) (a_1,a_2,\ldots,a_d).
\end{eqnarray*}
Since this vector is to be in the hyperplane $H$, we have:
\[ 1= a_d \sum_{i=1}^da_i .\]
\end{enumerate}
Combining this with (1) implies:
\[ a_d = \frac{\sum_{i=1}^{d-1}a_i^2}{\sum_{i=1}^{d-1}a_i}.\]
Now we can present the example:

\begin{example}
There are vectors $\{\phi_i\}_{i=1}^{2d-1}$ in $\RR^d$ which do
phase retrieval but $\{\phi_i^{\perp}\}_{i=1}^{2d-1}$ does not
do phase retrieval.
\end{example}

\begin{proof}
Consider the set
		 	$$A:=\left\{\left(a_1, a_2, \ldots, a_{d-1}, \dfrac{\sum_{i=1}^{d-1}a_i^2}{\sum_{i=1}^{d-1}a_i}\right) : a_i\in \mathbb{R}, \sum_{i=1}^{d-1}a_i\not=0 \right\}.$$
		 	Let $x=(1, 1, \ldots, 1)\in \RR^d$. Let any $\phi\in A$ and denote $P_{\phi}$ the orthogonal projection onto $\spn\{\phi\}$. Then we have
		 	$$(I-P_{\phi})(x)=x-\langle x,\dfrac{\phi}{\Vert \phi\Vert}\rangle \dfrac{\phi}{\Vert \phi\Vert}.$$
		 	Denote $b_d$ the $d-$ coordinate of $(I-P_{\phi})(x)$, then
		 	\begin{align*}
		 	b_d&=1-\dfrac{1}{\Vert \phi\Vert^2}\left(\sum_{i=1}^{d-1}a_i+\dfrac{\sum_{i=1}^{d-1}a_i^2}{\sum_{i=1}^{d-1}a_i}\right)\dfrac{\sum_{i=1}^{d-1}a_i^2}{\sum_{i=1}^{d-1}a_i}\\
		 	&=1-\dfrac{1}{\Vert \phi\Vert^2}\left(\sum_{i=1}^{d-1}a_i^2+\left(\dfrac{\sum_{i=1}^{d-1}a_i^2}{\sum_{i=1}^{d-1}a_i}\right)^2\right)=0.
		 	\end{align*}
		 	Let
		 	\begin{align*}
		 	\phi_1&=(1, 0, \ldots, 0, 1)\\
		 	\phi_2&=(0, 1, \ldots, 0, 1)\\
		 	&\cdots\\
		 	\phi_{d-1}&=(0, 0, \ldots, 1, 1)\\
		 	\phi_d=x&=(1, 1, \ldots, 1, 1).\\
		 		\end{align*}
		 		Then $\{\phi_i\}_{i=1}^d$ is a linearly independent set in $\RR^d$ and $\{\phi_i\}_{i=1}^d\subset A$.
		 		
		 		Now we will show that for any finite hyperplanes $\{W_i\}_{i=1}^k$ in $\RR^d$, there exists a vector $\phi\in A$ such that $\phi\notin \cup_{i=1}^kW_i$.
		 		Suppose by a contradiction that $A\subset \cup_{i=1}^kW_i$.
		 		
		 		Consider the set
		 		$$B:=\left\{\left(x, x^2, \ldots, x^{d-2}, 1-\sum_{i=1}^{d-2}x^i, \sum_{i=1}^{d-2}x^{2i}+\left(1-\sum_{i=1}^{d-2}x^i\right)^2\right) : x\in \RR\right\},$$ then $B\subset A$.
		 		
		 		Hence $B\subset \cup_{i=1}^kW_i.$
		 		Therefore, there exists $j\in\{1, \ldots, k\}$ such that $W_j$ contains infinitely many vectors in $B$.
		 		
		 		Let $u=(u_1, u_2, \ldots, u_d)\in W_j^{\perp}, u\not=0$. Then we have
		 		
		 		$$\langle u, \phi_x\rangle =0$$
		 		for infinitely many $\phi_x\in B$.
		 		
		 		Thus,
		 		 $$\sum_{i=1}^{d-2}u_ix^i+u_{d-1}\left(1-\sum_{i=1}^{d-2}x^i\right)+u_d\left(\sum_{i=1}^{d-2}x^{2i}+\left(1-\sum_{i=1}^{d-2}x^i\right)^2\right)=0,$$ for infinitely many $x$.
		 		
		 		This implies $u_1=u_2=\cdots=u_d=0$, which is a contradiction.
		 		
		 		From above, we can pick $d-1$ vectors $\{\phi_i\}_{i={d+1}}^{2d-1}$ in $B$ such that $\{\phi_i\}_{i=1}^{2d-1}$ is a full spark of vectors in $\RR^d$. Thus, $\{\phi_i\}_{i=1}^{2d-1}$ does phase retrieval in $\RR^d$.
		 		
		 		Moveover, since $\spn\{(I-P_i)(x)\}_{i=1}^{2d-1}\not=\RR^d$ then $\{\phi_i^{\perp}\}_{i=1}^{2d-1}$ cannot do phase retrieval by Theorem \ref{edidin}.
\end{proof}

In general, if hyperplanes $\{W_i\}_{i=1}^n$ do phase retrieval in $\RR^d$, it does not ensure that the complement vectors do phase retrieval. The following is an example.

\begin{ex}
	There are 5 vectors $\{\phi_i\}_{i=1}^5$
	in $\RR^3$ which fail phase retrieval but their induced
	hyperplanes $\{\phi_i^{\perp}\}_{i=1}^5$ do phase retrieval.
\end{ex}
\begin{proof}
	
	In $\RR^3$, let
	\begin{align*}
	W_1&=\spn\{e_2, e_3\}\\
	W_2&=\spn\{e_1, e_3\}\\
	W_3&=\spn\{e_1+e_2, e_3\}\\
	W_4&=\spn\{e_1, e_2+e_3\}\\
	W_5&=\spn\{e_2, e_1+e_3\}.
	\end{align*}
	Let $P_i$ be the projection onto $W_i$. Then for any $x=(x_1, x_2, x_3)$, we have
	\begin{align*}
	P_1x&=(0, x_2, x_3)\\
	P_2x&=(x_1, 0, x_3)\\
	P_3x&=\left(\dfrac{x_1+x_2}{2},\dfrac{x_1+x_2}{2}, x_3\right)\\
	P_4x&=\left(x_1, \dfrac{x_2+x_3}{2}, \dfrac{x_2+x_3}{2}\right)\\
	P_5x&=\left(\dfrac{x_1+x_3}{2}, x_2, \dfrac{x_1+x_3}{2}\right).
	\end{align*}
	For any $x\not=0$, the rank of the matrix whose the rows are $P_ix$ equals $3$. Therefore, $\{P_ix\}_{i=1}^5$ spans $\RR^3$. By Theorem \ref{edidin}, $\{W_i\}_{i=1}^5$ does phase retrieval in $\RR^3$.
	
	We also have $$W_1^\perp=\spn\{e_1\}, \quad W_2^\perp=\spn\{e_2\}, \quad W_3^\perp=\spn\{u_3\},$$
	$$ W_4^\perp=\spn\{u_4\}, \quad W_5^\perp=\spn\{u_5\},$$ for some $u_3, u_4, u_5\in \RR^3$.
	
	Since $e_1, e_2, u_3 \perp e_3$ then $\spn\{e_1, e_2, u_3\}\not=\RR^3$. Thus, $\{e_1, e_2, u_3, u_4, u_5\}$ fails the complement property. Therefore, it cannot do phase retrieval.
\end{proof}

\section{An Example in $\RR^4$}

In this section we will give an example of 6 hyperplanes in
$\RR^4$ which do phase retrieval.  First, we will show that
this is the minimal number of hyperplanes which can do
phase retrieval.

\begin{theorem}
	If hyperplanes $\{W_i\}_{i=1}^n$ do phase retrieval in $\RR^d$
	then $n\ge 2d-2$.  Moreover, if $n=2d-2$ then the vectors
	$\{W_i^{\perp}\}_{i=1}^{2d-2}$ are full spark.
\end{theorem}
\begin{proof}
	Assume, by way of contradiction, that $n \le 2d-3$.  Choose a
	vector
	\[ 0\not= x \in \cap_{i=1}^{d-1}W_i.\]
	So $P_ix=x$ for all $i=1,2,\ldots,d-1$.  It follows that the
	set $\{P_ix\}_{i=1}^n$ has at most $d-1$ non-zero vectors and
	hence cannot span $\RR^d$, contradicting Theorem \ref{edidin}.
	
\vskip12pt
For the {\it moreover} part, we proceed by way of contradiction.
Let $W_i^{\perp}=\spn\{\phi_i\}$ for $i=1,2,\ldots,2d-2$ and assume there
exists $I\subset [2d-2]$ with $|I|=d$ and $\{\phi_i\}_{i\in I}$
does not span the whole space.  Choose $0\not= x\perp \phi_i$
for all $i\in I$. It follows that $x\in W_i$ for all $i\in I$ and
so $P_ix=x$ for all $i\in I$.  But, $|I^c|=d-2$ and so $\{P_ix\}_{i=1}^{2d-2}$ contains at most $d-1$ distinct elements and so
cannot span, contradicting Theorem \ref{edidin}.
\end{proof}

Now we are ready for the main result of this section. In \cite{XU}
it was shown that there are six 2-dimensional subspaces of
$\RR^4$ which do phase retrieval.  We will now extend this result
to hyperplanes in $\RR^4$.

\begin{theorem}
	Suppose that $d=4$. There exist $6$ hyperplanes $W_1,\ldots,W_6\subset \RR^4$ which do phase retrieval on $\RR^4$.
\end{theorem}
\begin{proof}
	Set
	\[
	W_j:=\phi_j^\perp\subset \RR^4, \quad j=1,\ldots,6,
	\]
	where
	\begin{eqnarray*}
		\phi_1&=&(2,-1,2,2)/\sqrt{13},\,\,\phi_2=(2,5,4,1)/\sqrt{46},\,\, \phi_3=(0,4,-1,-1)/\sqrt{18},\\
		\phi_4&=&(5,4,-2,-4)/\sqrt{61},\,\, \phi_5=(4,1,5,3)/\sqrt{51}, \,\, \phi_6=(3,-4,-4,-3)/\sqrt{50}.
	\end{eqnarray*}
	
Note that $\|P_jx\|^2={\rm Tr}(A_j X)$ where $X=xx^T$, $A_j=u_ju_j^T+v_jv_j^T+\omega_j\omega_j^T$
and $u_j,v_j,\omega_j\in \RR^4$ is an  orthonormal  basis of $W_j$. Then $W_1,\ldots,W_6$ do phase retrieval if and only if
\[
{\mathcal Z}:=\{Q\in \RR^4: Q=Q^T, {\rm rank}(Q)\leq 2, Tr(A_jQ)=0, j=1,\ldots,6\}
\]
only contains zero matrix.
We write $Q$ in the form of
\[
Q= \begin{pmatrix}
x_{11} & x_{12} & x_{13}& x_{14} \\
x_{12} & x_{22} & x_{23} & x_{24}\\
x_{13} & x_{23} & x_{33} & x_{34} \\
x_{14} & x_{24} & x_{34} & x_{44} \\
\end{pmatrix},
\]
where $x_{jk}, 1\leq j\leq k\leq 4,$ are $10$  variables. The ${\rm rank}(Q)\leq 2$ if and only if
$m_{j,k}:=m_{j,k}(x_{11},x_{12},\ldots,x_{44})=0$ where $m_{j,k}$  denotes the determinant of the sub-matrix  formed by deleting the $j$th row and $k$th column from the matrix $Q$. Noting that $A_j=I-\phi_j\phi_j^T$, we have
\[
\ell_j:=\ell_j(x_{11},\ldots,x_{44}):=Tr(A_jQ)=Tr(Q)-\phi_j^TQ\phi_j.
\]
The ${\mathcal Z}$ only contain zero matrix if and only if  the homogeneous  polynomial system
\begin{equation}\label{eq:deng}
\ell_1=\cdots=\ell_6=m_{1,1}=\cdots=m_{4,4}=0
\end{equation}
has no  non-trivial (i.e., non-zero) real solutions.

We next verify the polynomial system (\ref{eq:deng}) only has real zero solution following the ideas of  Vinzant \cite[Theorem 1]{small}.
Using the computer algebra software {\tt Maple}, we compute   a Gr\"{o}bner basis of the ideal
\[
\left<\ell_1,\ldots,\ell_6,m_{1,1},\ldots,m_{4,4}\right>
\]
 and elimination (see \cite{Sturm}).
 The result is a polynomial $f_0\in \Q[x_{34},x_{44}]$, which is a homogeneous polynomial of degree $10$:
	
	 \begin{eqnarray*}
	 	f_0&=& 615836814694440125755941750205355957259806055430532973956877900 x_{4,4}^{10}\\
	& &-884972594452387958848562473144241797030697764519228205098183524x_{4, 4}^9x_{3, 4}\\
	 	& &+37549510562762689603032479610577980614684970115180508761212602923x_{4, 4}^8x_{3, 4}^2\\
	 	& &-261784289245252068342511157673868998003077035922935758454568869970x_{4, 4}^7x_{3, 4}^3\\
	 	& &+1318646361014374203805595493716801537462083922918839965435901151518x_{4, 4}^6x_{3, 4}^4\\
	 	& &+2323672503729013471271218611541822606087314313103855222266887257194x_{4, 4}^5x_{3, 4}^5\\
	 	& &+841099655929202539990506870648349938942927420225588274968467286492x_{4, 4}^4x_{3, 4}^6\\
	 	& &+2453118466138743624272476494499733256382267234695398509857315458204x_{4, 4}^3x_{3, 4}^7\\
	 	& &+2686702635361560203562012680667911834582476444588124478311966009776x_{4, 4}^2x_{3, 4}^8\\
	 	& &+59872475066978406270800582425071592403273130463063552339405262912x_{4, 4}x_{3, 4}^9\\
	 	& &+950484050032900617743793729374383632917614227356173754905368787200x_{3, 4}^{10}.
	 \end{eqnarray*}
We can verify that the univariate polynomial $f_0(1,x_{4,4})$ has no real zero roots  using Sturm sequence, and hence
\[
\{(x_{3,4},x_{4,4})\in \RR^2 : f_0(x_{3,4},x_{4,4})=0\}=\{(0,0)\},
\]
which implies that  if $(x_{1,1},x_{1,2},\ldots,x_{3,4},x_{4,4})$
is a real solution of (\ref{eq:deng})  then $x_{3,4}=x_{4,4}=0$.
 By computing  a Gr\"{o}bner basis of the ideal, we obtain that
\[
1\in \left<x_{3,4},x_{4,4}, x_{j,k}-1,\ell_1,\ldots,\ell_6,m_{1,1},\ldots,m_{4,4}\right>, \quad 1\leq j\leq k\leq 4
\]
which means that
(\ref{eq:deng}) does not have nonzero real root with $x_{3,4}=x_{4,4}=0$.
The maple code for these computation is posted at http://lsec.cc.ac.cn/$\sim$xuzq/phase.htm.
 Combining results above, we obtain that
(\ref{eq:deng}) has no  non-trivial  real solutions.

\end{proof}

\begin{corollary}
There are six hyperplanes $\{W_i\}_{i=1}^6$ doing phase retrieval on $\RR^4$ but
$\{W_i^{\perp}\}_{i=1}^6$ does not do phase retrieval.
\end{corollary}

\bibliographystyle{amsplain}

\begin{thebibliography}{WW}
	
	\bibitem{BCE}  R. Balan, P.G. Casazza, and D. Edidin,
	{\it On Signal Reconstruction Without Phase}, Appl.  Comput. Harmon. Anal., {\bf 20} (3) (2006) 345-356.
	
	\bibitem{BCMN}  A.S. Bandeira, J. Cahill, D. Mixon and A.A.
	Nelson, {\it Saving phase: injectivity and stability for
		phase retrieval}, Appl.  Comput. Harmon. Anal., {\bf 37} (1) (2014) 106-125.
	
	\bibitem{BaMn86}
	R. H. Bates and D. Mnyama, {\it The status of practical Fourier phase retrieval}, Advances in Electronics and Electron Physics, {\bf 67} (1986), 1-64.
	
	
	\bibitem{BeRi99}
	C. Becchetti and L. P. Ricotti, {\it Speech recognition theory and C++ implementation}, Wiley (1999).
	
	\bibitem{B}  B. Bodmann and N. Hammen, \emph{Stable Phase
		retrieval with low redundancy frames}, Preprint. arXiv:1302.5487.
	
	\bibitem{SaCa016}
	S. Botelho-Andrade, P. G. Casazza, H. Van Nguyen, J. C. Tremain, {\it Phase retrieval versus phaseless reconstruction}, J. Math Anal. Appl., {\bf 436} (1), (2016) 131-137.
	
	\bibitem{BaCaCaJaWo014}
	J. Cahill, P. G. Casazza, J. Jasper, and L.M. Woodland, {\it Phase retrieval and norm retrieval}, (2014). arXiv preprint arXiv:1409.8266.
	
	
	\bibitem{CCPW}
	J. Cahill, P. Casazza, K. Peterson, L. Woodland,
	\emph{Phase retrieval by projections},
	Available online: arXiv:1305.6226.
	
	\bibitem{CEHV}  A. Conca, D. Edidin, M. Hering, and C.
	Vinzant, {\it An algebraic characterization of injectivity of phase retrieval}, Appl. Comput. Harmon. Anal., {\bf 38} (2)
	(2015) 346-356.
	
	\bibitem{ED}  D. Edidin, {\it Projections and phase retrieval},
	 Appl. Comput. Harmon. Anal.,
{\bf 42} (2)(2017) 350-359.

\bibitem{Dr010}
    J. Drenth, {\it Principles of protein x-ray crystallography}, Springer, 2010.
	
	\bibitem{Fi78}
	J. R. Fienup, {\it Reconstruction of an object from the modulus of its fourier transform}, Optics Letters, {\bf 3} (1978), 27-29.
	
	\bibitem{Fi82}
	J. R. Fienup, {\it Phase retrieval algorithms: A comparison}, Applied Optics, {\bf 21} (15) (1982), 2758-2768.
	
	
	
	\bibitem{HMW}  T. Heinosaari, L. Maszzarella, and M.M. Wolf, \emph{Quantum tomography under prior information}, Comm. Math. Phys. {\bf 318} No. 2 (2013) 355-374.
	
	\bibitem{RaJu93}
	L. Rabiner, and B. H. Juang, {\it Fundamentals of speech recognition}, Prentice Hall Signal Processing Series (1993).
	
	\bibitem{ReBlScCa004}
	J. M. Renes, R. Blume-Kohout, A. J. Scott, and C. M. Caves, {\it Symmetric Informationally Complete Quantum Measurements}, J. Math. Phys., {\bf 45} (2004), 2171-2180.
	
	\bibitem{Sturm} B. Sturmfels, What is ... a Gr\"{o}bner basis?,  Notices Amer. Math. Soc, 52(10) (2005)  1199-1200.
	
	\bibitem{small}Cynthia Vinzant, A small frame and a certificate of its injectivity, Sampling Theory and Applications (SampTA) Conference Proceedings. (2015)197-200.
	
	

	
\bibitem{WX}Yang Wang, Zhiqiang Xu, Generalized phase retrieval : measurement number, matrix recovery and beyond, Available online: arXiv:1605.08034.
	
	\bibitem{XU}  Z. Xu, {\it The minimal measurement number for low-rank
		matrix recovery}, Appl.  Comput. Harmon. Anal. (2017), http://dx.doi.org/10.1016/j.acha.2017.01.005 .
	
\end{thebibliography}

\end{document}